\documentclass[a4paper,12pt]{amsart}
\usepackage{amsmath,amssymb}
\newtheorem{thm}{Theorem}
\newtheorem{prop}[thm]{Proposition}
\def\B{\mathbb B}
\def\C{\mathbb C}
\def\CC{\mathcal C}
\def\var{\varepsilon}
\def\ds{\displaystyle}

\title{On the squeezing function and Fridman invariants}

\author{Nikolai Nikolov and Kaushal Verma}

\address{NN: Institute of Mathematics and Informatics, Bulgarian Academy of
Sciences, Acad. G. Bonchev 8, 1113 Sofia, Bulgaria \newline\indent
Faculty of Information Sciences, State University of Library Studies
and Information Technologies, Shipchenski prohod 69A, 1574 Sofia, Bulgaria}
\email{nik@math.bas.bg}

\address{KV: Department of Mathematics, Indian Institute of Science, Bangalore 560 012, India}
\email{kverma@iisc.ac.in}

\thanks{The first named author is partially supported by the Bulgarian National Science Fund,
Ministry of Education and Science of Bulgaria under contract DN 12/2.}

\begin{document}

\begin{abstract}
For a domain $D \subset \mathbb C^n$, the relationship between the squeezing function and the Fridman
invariants is clarified. Furthermore, localization properties of these functions are obtained.
As applications, some known results concerning their boundary behavior are extended.
\end{abstract}

\maketitle

\section{Relating $s_D,$ $h^c_D$ and $h^k_D$}

\noindent Let $D\subset\mathbb{C}^n$ be a domain. The first purpose of this note is to clarify the relationship between $s_D$,
the squeezing function \cite{DGZ1}, and its dual counterparts -- the Fridman invariants $h^c_D$ and $h^k_D$ \cite{Fr1,Fr2}:
$$s_D(z)=\sup\{r : r \B_n \subset f(D), f\in\mathcal O(D,\B_n), f(z) = 0, f\text{ is injective}\},$$
and
$$h^d_D(z) = \sup\{\tanh r: B^d_D(z, r)\subset f(\B_n), f\in\mathcal O(\B_n, D), f\text{ is injective}\},$$
where $\B_n\subset\C^n$ is the unit ball, $d_D$ is the Carath\'{e}odory/Kobayashi pseudodistance $c_D/k_D,$
and $B^d_D(z, r)$ is the $d_D$-ball centered at $z$ with radius $r$ (in \cite{Fr1,Fr2}, $\sup\tanh r$ is replaced by $\inf 1/r$).
\smallskip

\noindent{\it Remark.} We set $s_D=0$ if $D$ is not biholomorphic to a bounded domain. We also allow $h^d_D(z)=0.$
\smallskip

Many properties and applications of $s_D$ have been recently explored by various authors,
see e.g. \cite{DGZ1,DGZ2,FW1,FW2,Nik,Z1,Z2,Z3} and the references therein; for $h^c_D$ and $h^k_D$ see \cite{Fr1,Fr2,MV}.

Let $\Delta$ be the unit disc. Recall that
$$c_D (z,w) = \sup \{\tanh^{-1}|\lambda|: f \in \mathcal O(D, \Delta), f(z) = 0, f(w)=\lambda \},$$
and $k_D$ is the largest pseudodistance that does not exceed the Lempert function $l_D:$
$$l_D (z,w)=\inf \{\tanh^{-1}|\lambda|: f \in\mathcal O(\Delta, D), f(0) = z, f(\lambda) = w \}.$$

The construction of $s_D,$ $h^c_D$ and $h^k_D$ suggests that there should be a natural relation between them.
This is indeed the case.

\begin{prop}\label{ineq}
Let $D\subset \mathbb C^n$ be a domain. Then
$$s_D\le h^c_D\le h^k_D.$$
\end{prop}

\begin{proof} The second inequality follows from the fact that $c_D\le k_D.$

\medskip

To prove the first inequality, we may assume that $s_D>0.$ By normal family arguments,
there exists an extremal mapping $f\in\mathcal O(D, \B_n)$ for $s_D(p) = r$ \cite[Theorem 2.1]{DGZ1}).
Then $g(z)=f^{-1}(rz)$ is a competitor for $h^c_D(p).$ If $q\in B^c_D(p, \tanh^{-1}r),$ then
$$\tanh^{-1}r>c_D(p, q)=k_{f(D)}(0,f(q))\ge c_{\B_n}(0,f(q))$$
which means that $f(q)\in r\B_n$, that is, $q \in g(\B_n).$ So
$$B^c_D(p,\tanh^{-1}r)\subset g(\B^n)\subset D$$
and hence $s_D(p)\le h^c_D(p).$
\end{proof}

Since the supremum in the definition of $s_D(z)>0$ is attained, then
$s_D(z)=1$ implies that $D$ is biholomorphic to $\B_n$ (in short, $D\sim\B_n$).
The last is also a consequence of Proposition \ref{ineq}
and the fact that if $D$ is $k$-hyperbolic, then $h^k_D(z)=1\Rightarrow D\sim\B_n$
\cite[Theorem 1.3(2)]{Fr2}).

\medskip

Note that $s_D$ is a continuous function \cite[Theorem 3.1]{DGZ1}). It is proved in
\cite[Theorem 1.3(1)]{Fr2} that $h^k_D$ is continuous if $D$ is $k$-hyperbolic.
In fact, more is true: $h^d_D$ is continuous for any continuous pseudodistance $d_D.$
This is a simple consequence of the observation that $B^d_D(z,r)\subset B^d_D(z,r+d_D(z,w)).$
In particular, $h_D^k$ and $h_D^c$ are always continuous.

\section{Localization of $s_D$ and $h^k_D$}

In the next section, we need two localization type results for $s_D$ and $h^k_D.$ These results
are in opposite directions.

\begin{prop}\label{locs} Let $D \subset \mathbb C^n$ be a domain and let $U$ be a neighborhood
of a plurisubharmonic peak point\footnote{This notion has a local character (see e.g. the proof
of \cite[Lemma 2.1.1]{Gau}). Hence any local holomorphic peak point is a plurisubharmonic peak point.}
$p_0\in\partial D.$ Then $$\lim_{z\to p_0}s_D(z)=1\Rightarrow\lim_{z\to p_0}s_{D\cap U}(z)=1.$$
\end{prop}

\noindent{\it Remark.} The inverse implication cannot be true without global assump\-tions about $D.$

\begin{proof} Let $t\in(0,1).$ It follows by the proof of \cite[Lemma 2.1.1]{Gau}
that there exists a neighborhood $V$ of $p_0$ such that
if $\psi\in\mathcal O(\Delta,D)$ and $\psi(0)=z\in D\cap V,$ then
$\psi(t\Delta)\subset D\cap U.$ In particular,
\begin{equation}\label{3}
l_D(z,w)\ge\tanh^{-1}t,\quad z\in D\cap V,\ w\in D\setminus U.
\end{equation}

Let now $f\in\mathcal O(D,\B_n)$ be an extremal mapping for $s_D(z)=r,$
$z\in D\cap V.$ If $w\in f^{-1}(tr\B_n),$ then
$$l_D(z,w)=l_D(0,f(w))\le l_{r\B_n}(0,f(w))<\tanh^{-1}t.$$
This and \eqref{3} show that $f(tr\B_n)\subset U$ and hence
$$s_{D\cap U}(z)\ge t.s_D(z)$$
which implies the desired result.
\end{proof}

\begin{prop}\label{lock} Let $D \subset \mathbb C^n$ be a domain and let $U$ be a neighborhood
of a local holomorphic peak point $p_0\in\partial D.$ Then
$$\lim_{z\to p_0}h^k_{D\cap U}(z)=1\Rightarrow\lim_{z\to p_0}h^k_D(z)=1.$$
\end{prop}

For the strongly pseudoconvex case, this proposition is contained in \ref{lock}=\cite[Lemma 3.4]{Fr2}.

\begin{proof}  Denote by $\kappa_D$ the Kobayshi--Royden pseudometric:
$$\kappa_D (z;X)=\inf\{|\lambda|: f \in \mathcal O(\Delta, D), f(0) = z, \lambda f'(0) = X \}.$$

Let $s\in(0,1)$ and $r>0.$ It follows by the proofs of \cite[Lemmas 2.1.1 and 2.1.3]{Gau}
that there exists neighborhoods $U_2\subset U_1\subset U$ such that
\begin{equation}\label{1}
s\kappa_{D\cap U}(u;X)\le\kappa_D(u;X),\quad u\in D\cap U_1,\ X\in\C^n,
\end{equation}
and
\begin{equation}\label{2}k_D(z,w)\ge r,\quad z\in D\cap U_2,\ w\in D\setminus U_1.
\end{equation}

Let $z\in D\cap U_2$ and $w\in B^k_D(z,r).$ Since $k_D$ is the integrated form of $\kappa_D,$ there is
a smooth curve $\gamma:[0,1]\to D$ joining $z$ and $w$ such that
$$\int_{0}^1\kappa_D(\gamma(t);\gamma'(t))dt<r.$$
By \eqref{2}, $\gamma[0,1]\subset U_1$ and then \eqref{1} implies that $w\in B^k_{D\cap U}(z,r/s).$
So
$$B^k_D(z,r)\subset B^k_{D\cap U}(z,r/s)$$
and the desired result follows easily.
\end{proof}

\section{Detecting strong pseudoconvexity}

Applying \cite[Lemma 3.1]{FW1}, it follows by the proof of \cite[Theorem 1.3]{DGZ2}
that if $p_0$ is a strongly pseudoconvex boundary point of a bounded domain $D\subset\C^n$ such that
$\overline D$ admits a Stein neighborhood basis, then
$$\lim_{z\to p_0}s_D(z)=1.$$

Combining this fact with Proposition \ref{ineq} and Proposition \ref{lock},
we get the following result (see also \cite[Theorem 1.1(i)]{MV} and, in the $\CC^3$-smooth setting,
\cite[Theorem 3.1]{Fr2}).

\begin{prop}\label{S} If $p_0$ is a strongly pseudoconvex boundary point of a domain $D\subset\C^n,$ then
$$\lim_{z\to p_0}h^k_D(z)=1.$$
\end{prop}

Conversely, it turns out that strong pseudoconvexity can be characterized by the boundary
behavior of the squeezing function. For example, if $D\subset\C^n$
is a bounded convex domain with $\CC^{2,\alpha}$-smooth
boundary, then \cite[Theorem 1.7]{Z3}
$$\lim_{z\to\partial D}s_D(z)=1\Rightarrow D\mbox{ is strongly pseudoconvex}$$
($\CC^2$-smoothness is not enough as \cite[Theorem 1.1]{FW2} shows).

\medskip

Also, if $p_0$ is an $h$-extendible boundary point of a $\CC^\infty$-smooth bounded
pseudoconvex domain $D\subset\C^n,$
then \cite[Theorem 1]{Nik}
$$\lim_{z\to\partial D}s_D(z)=1\Rightarrow p_0\mbox{ is strongly pseudoconvex}.$$

Recall that a $\CC^\infty$-smooth pseudoconvex boundary point $p_0$ of finite type of a
domain $D\subset\mathbb{C}^n$ is said to be $h$-extendible (see e.g. \cite{Yu} and the references
therein) if $D$ is pseudoconvex near $p_0$ and the Catlin and D'Angelo multitypes of $p_0$ coincide.
For example, $p_0$ is $h$-extendible if the Levi form at $p_0$ has a corank at most one, or
$D$ is linearly convexifiable near $p_0$. In particular, $h$-extendibility takes place
in the strictly pseudoconvex, two-dimensional finite type, and convex finite type cases.

\medskip

The results below localize \cite[Theorem 1]{Nik} and, in some cases, \cite[Theorem 1.7]{Z3}.

\begin{prop}\label{H}
Let $p_0$ be an $h$-extendible boundary point of a domain $D\subset\mathbb{C}^n.$
If $\lim_{z\to p_0}s_{D}(z)=1$, then $\partial D$ is strongly pseudoconvex at $p_0.$
\end{prop}

\begin{proof} Since $p_0$ is a local holomorphic peak point \cite[Theorem A]{Yu},
the result follows by choosing a neighborhood $U$ of $p_0$ such that $D\cap U$ to be a $\CC^\infty$-smooth
pseudoconvex bounded domain, and then applying Proposition \ref{locs} and Theorem \cite[Theorem 1]{Nik}.
\end{proof}

\begin{prop}\label{C} Let $D\subset\C^n$ be a domain that is locally convexifiable near
a $\CC^\infty$-smooth boundary point $p_0,$ and $\partial D$ contains no analytic discs
through $p_0.$ If $\lim_{z\to p_0}s_D(z)=1$, then $\partial D$ is strongly pseudoconvex at $ p_0 $.
\end{prop}

\begin{proof}
If $p_0$ is of finite type, then the result follows from Proposition \ref{H}.

\medskip

Suppose that $p_0$ is of infinite type. By the disc assumption, $p_0$ is again a local holomorphic peak point
\cite[Proposition 3.4]{Sib}. Then, using Proposition \ref{locs},
we may assume that $D$ is a convex domain. It follows by the proof of \cite[Propoisition 6.1]{Z1}
that there exists a sequence $(z_j)\subset D$ converging to $p_0$ and affine isomorphisms $A_j$ of
$\C^n$ such that $A_j(D)\to D_\infty$ in the local Hausdorff topology, and
$A_j(z_j)\to 0\in D_\infty,$ where $D_\infty$ is a convex domain, containing
no complex affine lines (by the disc assumption), and $\partial D_\infty$ contains an affine disc $\mathcal D$
(by the infinite type assumption). Since $s_{A_j(D)}(A_j(z_j))=s_D(z_j)\to 1,$ then $s_{D_\infty}(0)=1$
by \cite[Proposition 7.1]{Z2}, that is, $D_\infty\sim\B^n.$ In particular, the metric space
$(D_\infty,k_{D_\infty})$ is Gromov hyperbolic which is a contradiction to $\mathcal D\in\partial D_\infty$
\cite[Proposition 1.10=Proposition 6.1]{Z1}.
\end{proof}

\section{Estimates of $s_D$ and $h^k_D$}

When $\partial D$ has higher regularity than $\CC^2$ near a strongly
pseudoconvex point $p_0,$ quantitative lower bounds for $s_D$ hold.

Set $\delta_D(z)=\mbox{dist}(z,\partial D).$
Let $\var\in (0,1],$ $m\in\{0,1\}$ and $\ds\var_k=\frac{m+\var}{2}.$

\begin{prop}\label{ests}\cite[Theorem 1]{NT} Let $p_0$ be a $\CC^{m+2,\var}$-smooth
strongly pseudoconvex boundary point of a bounded domain $D\subset\C^n$
such that $\overline D$ admits a Stein neighborhood basis. Then
$$\limsup_{z\to p_0}\frac{1-s_D(z)}{\delta_D(z)^{\var_m}}<\infty.$$
\end{prop}

The proof of \cite[Theorem 1]{NT} follows the same lines as that of
\cite[Theorem 1.1]{FW1}, where the cases of $\CC^3$-smooth
and $\CC^4$-smooth bounded strongly pseudoconvex domains are considered.

In the $\CC^{3,1}$-smooth case, Proposition \ref{ests} says that
$$\limsup_{z\to p_0}\frac{1-s_D(z)}{\delta_D(z)}<\infty.$$

This estimate is optimal. Indeed, it follows from the proof of \cite[Theorem 1.2]{FW2} that
if $p_0$ is a $\CC^2$-smooth boundary point of a domain $D\subset\C^n$ and
$$\liminf_{z\to p_0}\frac{1-s_D(z)}{\delta_D(z)}=0,$$
then $D\sim\B_n.$

We will prove more in a short way.

\begin{prop}\label{ball} Let $p_0$ be a Dini-smooth boundary point\footnote{For this notation see
e.g. \cite[p. 45]{NA}. Note that $\CC^{1,\var}$- $\Rightarrow$ Dini-
$\Rightarrow$ $\CC^1$-smoothness.}
of a domain $D\subset\C^n.$ If
$$\liminf_{z\to p_0}\frac{1-s_D(z)}{\delta_D(z)}=0,$$
then $D\sim\B_n.$
\end{prop}

\begin{proof} One may find points $z_j\to p_0$ and injective maps $f_j\in\mathcal O(\B_n,D)$
such that
\[
B_D^k(z_j,r_j)\subset f_j(\B_n)
\]
and
$$\frac{1-\tanh r_j}{\delta_D(z_j)}\to 0.$$ Fix an $a\in D.$
By \cite[Theorem 7]{NA}, there exists a constant $c>0$ with
$$2k_D(a,z_j)\le-\log\delta_D(z_j)+c.$$
Then $r'_j=r_j-k_D(a,z_j)\to\infty$ and $B_D^k(a,r'_j)\subset B_D^k(z_j,r_j)\subset f_j(\B_n).$
It follows that $h^k_D(a)=$ and hence $D\sim B_n.$
\end{proof}

Finally, note that Proposition \ref{ineq} implies that the estimate from Proposition \ref{ests} holds
for $h^k_D,$ too. It turns out that such an estimate for $h^k_D$ is valid without global assumptions about
$D$ (which is not true for $s_D$).

\begin{prop}\label{esth}\cite[Theorem 1]{NT} Let $p_0$ be a $\CC^{m+2,\var}$-smooth
strongly pseudoconvex boundary point of a domain $D\subset\C^n.$ Then
$$\limsup_{z\to p_0}\frac{1-h^k_D(z)}{\delta_D(z)^{\var_m}}<\infty.$$
\end{prop}

\begin{proof} It follows by the proof of \cite[Theorem 1]{BB} that the behavior of $k_D$ near $p_0$
is determined up to small/large additive constants by the local geometry of $\partial D$ near $p_0.$
Then there exist neighborhoods $V\subset U$ of $p_0$ and a constant $c>0$
such that $D\cap U$ to be a strongly pseudoconvex domain and
$$k_{D\cap U}(z,w)\le k_D(z,w)+c,\quad z,w\in D\cap V.$$
On the other hand, we may shrink $V$ and enlarge $c$ such that
$$\tanh^{-1}h^k_{D\cap U}(z)>r(z):=-\frac{1}{2}\log\delta_D(z)-c,\quad z\in D\cap V$$
(by Propositions \ref{ineq} and \ref{ests}). Increasing $c$ once more, one may find
a neigh\-borhood $W\subset V$ of $p_0$ such that
$$k_D(z,w)\ge r(z)-c,\quad z\in D\cap W,\ w\in D\setminus V$$
(by \cite[2.3. Theorem]{FR}). Then
$$B^k_D(z,r(z)-c)\subset B^k_{D\cap U}(z,r(z)),\quad w\in D\cap W,$$
which implies the desired inequality.\end{proof}

\end{document}